\documentclass[12pt,a4paper,english,leqno]{article}
\usepackage[T1]{fontenc}
\usepackage{verbatim}
\usepackage{amsmath}
\usepackage{graphicx}
\usepackage{amssymb}
\usepackage{esint}

\makeatletter

\pdfpageheight\paperheight
\pdfpagewidth\paperwidth

\usepackage{mathrsfs}\usepackage{amsthm}
\usepackage{longtable}
\usepackage{multirow}
\usepackage{hhline}

\@ifundefined{definecolor}
 {\@ifundefined{definecolor}
 {\usepackage{color}}{}
}{}

\definecolor{BlueGreen}{RGB}{49,152,255}
\definecolor{Violet}{RGB}{120,80,120}
\definecolor{Olive}{RGB}{128, 128, 0}
\definecolor{MidnightBlue}{RGB}{25, 25, 112}

\makeatletter

\AtBeginDocument{\DeclareFontEncoding{T2A}{}{}}

\usepackage{eufrak}\usepackage{amsfonts}\usepackage{graphics}\usepackage{floatflt}

\usepackage{amsthm}

\newtheorem{corollary}{Corollary}

\makeatother

\usepackage{cite}

\makeatother

\usepackage{babel}

\begin{document}
\newtheorem{utv}{Lemma}\newtheorem{theorem}{Theorem}\newtheorem*{myth}{Theorem}
\global\long\global\long\def\GmLn#1#2{G_{#1}(t;#2)}
 \global\long\global\long\def\mExpLn#1#2{(#2-t)^{#1}e^{i\alpha\log|#2-t|}}
 \global\long\global\long\def\uExpLn#1{e^{i\alpha\,\log|#1-t|}}
\global\long\def\Bounded#1#2{\Xi_{#2}(#1)}

\global\long\global\long\def\mysz{{\textstyle }}
\renewcommand{\ln}{\log}



\begin{center}
{\Large On properties of optimal controls for~an~inverted spherical
pendulum}
\par\end{center}{\Large \par}

\begin{center}
L.~Manita$^{1}$ and M.~Ronzhina$^{2,3}$
\par\end{center}

\begin{center}
{\small $^{1}$ NRU Higher School of Economics, Moscow State Institute
of Electronics and Mathematics, Moscow, Russia }\\
{\small e-mail: lmanita@hse.ru}
\par\end{center}{\small \par}

\begin{center}
{\small $^{2}$ Gubkin Russian State University of Oil and Gas (National
Research University), }\\
{\small $^{3}$ NRU Higher School of Economics, Faculty of Computer
Science, Moscow, Russia}\\
{\small e-mail: maryaronzhina@gmail.com}
\par\end{center}{\small \par}
\begin{abstract}
In this paper we study an  optimal control problem that is affine
in two-dimensional bounded control. The problem is related to the
stabilization  of an inverted spherical pendulum in the vicinity
of the upper unstable equilibrium. We find solutions stabilizing the
pendulum in a finite time, wherein the corresponding optimal controls
perform an infinite number of rotations along the circle $S^{1}$.
\end{abstract}

\section{Introduction}

{\let\thefootnote\relax\footnote{

This research was supported in part by the Russian Foundation for
Basic Research under grant no. 17-01-00805.

Key words: linear-quadratic optimal control problem, hamiltonian system,
singular extremal, blow-up singularity, stabilization, inverted spherical
pendulum

Mathematics Subject Classification: $34H05,34H15,49N10,49N90$

}}\vspace*{-1.\baselineskip}

 Models of inverted pendulum systems are widely used to study the
dynamics of different complex nonlinear objects in robotics, mechanics,
aerospace engineering, personal transport systems \cite{japan,Nicolosi,humanoid-robot,Precup-2008}.
 For pendulums, the upper vertical position is obviously unstable.
However, it was proved that it is possible to turn this position into
a stable one, for example, if the suspension point of a planar pendulum
performs vertical oscillations \cite{Kapitza,Arnold} or moves along
a horizontal line (e.g. \cite{Formal-2006,mart-form2008,manita-ronzhina}).
 It turns out that motions of a spherical inverted pendulum can also
be stabilized by choosing an appropriate external control, e.g. \cite{Liu-2008-1,Gutierrez-1,Postelnik-1,tracking-2009-Cosolini}.
Many methods to control pendulum systems use an optimal control technique:
a stabilizing controller is derived from minimization of a quadratic
cost functional (LQR controllers). But applied to real physical systems,
such controllers may produce relatively large deviations of pendulum
systems from the upper equilibrium position. To improve this, there
are different approaches, i.e., in \cite{Nelson1994} it was proposed
to combine optimal and neural network techniques. In this paper, we
consider the stabilization problem of the inverted spherical pendulum
in terms of minimizing the mean square deviation of the pendulum from
the unstable upright equilibrium point over an infinite time interval.
  We assume that  the spherical inverted pendulum is on a movable
base which  moves in the horizontal plane under the influence of
a planar bounded force.  We study the behaviour of solutions for
the linearized model. The corresponding optimal control problem is
affine in two-dimensional bounded control. We show that  optimal
solutions of this problem exhibit complicated behaviour. This is due
to the fact that the upright equilibrium is a singular mode for the
problem under consideration.

Singular modes are characterized by the fact that over some open interval
the Hamiltonian reaches a maximum at more than one point, that is,
the optimal control is not determined directly from the maximality
condition of the Pontryagin maximum principle.   Singular solutions
 appear in many applications:  optimal spacecraft flights (intermediate
thrust arcs) or problems of spacecraft reorientation \cite{Robbins,Seywald1993,Park-space,Trelat-1-chattering,Trelat-2,Goh-2008},
robotics (controlling manipulators \cite{Zel-Bor1994,Manita2000},
the Dubins car problem \cite{Agrachev-Sachkov}), mathematical models
in economics \cite{Zelikin-2005,Yegorov-2015}, biomedical problems
\cite{Ledzewicz-2009,Ledzewicz-Sch-2012,Ledzewicz-book1,Yegorov}.
 For more details on singular solutions  see, for example, \cite{Kupka-1990,Zel-Bor1994,Bonnard-Chyba2003}. 

Often optimal trajectories consist of nonsingular and singular arcs
and the concatenation structures of these arcs  can be very irregular,
for example, the chattering or the Fuller phenomenon (an infinite
number of control discontinuities in a finite time interval) \cite{Zel-Bor1994,Zel-Manita-1,Ledzewicz-2009,Trelat-2,Shen-chattering,Robbins},
iterated Fuller singularities \cite{2019-Boarotto-iterated-Fuller-singularities},
a chaotic behaviour of bounded pieces of optimal trajectories \cite{Z-L-H_2017}.
Such structure of optimal controls, rather complicated from a mathematical
point of view, is very typical for controlling  systems that possess
singular regimes. 

In \cite{lob-1} it was proved that for some initial conditions the
optimal control problem for the inverted spherical pendulum  is reduced
to the problem with scalar bounded control. In this case there exist
optimal chattering solutions.   In the present paper,   for
initial positions of the pendulum close to upper equilibrium, we find
optimal solutions stabilizing the pendulum in the upper position in
a finite time. For this solutions  the corresponding optimal controls
perform an infinite number of rotations along the circle $S^{1}$.
 Note that most of the results for problems with singular solutions
and bounded multidimensional control were obtained when the control
set is a multi-dimensional rectangle or a polyhedron \cite{Ledzewicz-book2,Chyba-2003-Application-to-Controlled-Mechanical-Systems}.
The structure of optimal solutions for the control set, which is a
convex set but not a polyhedron, was studied, e.g., in \cite{Chukanov,Z-L-H_2017,Lokut-Myrikova-2019}.
This paper is a extension of the result \cite{Chukanov,Zel-Bor1994}
obtained for homogeneous problem in which the control set is the unit
disc to the nonhomogeneous case.

\section{Problem Formulation }

\begin{figure}
\begin{centering}
\includegraphics[width=0.95\textwidth]{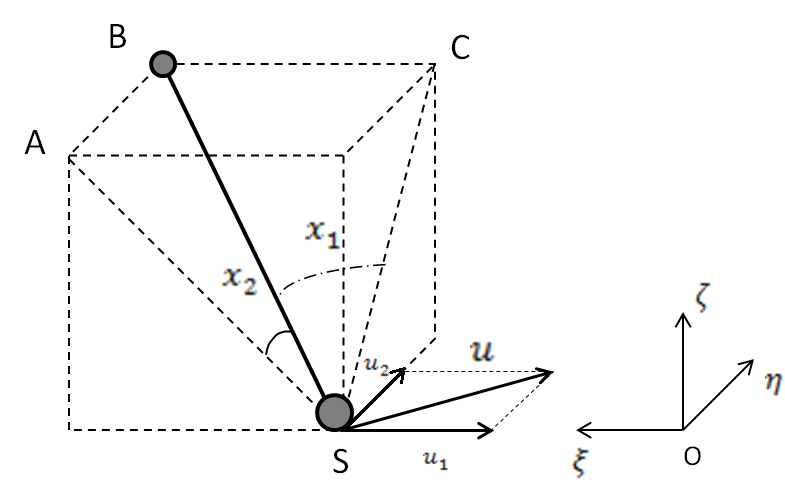}
\par\end{centering}

\caption{the inverted spherical pendulum}
\end{figure}
We consider a mathematical model of the spherical inverted pendulum
with a moving support point. The pendulum consists of the point mass
$B$ on the end of the rigid massless rod of length $l$ (Fig.~1).
The rod is attached by a hinge to the moving support point $S$. We
suppose that there is no friction in the hinge. The support point
$S$ can move in the horizontal plane under the action of an external
planar control force $u=\left(u_{1},u_{2}\right)\in\mathbb{R}^{2}$.
The control force is assumed to be bounded: $\,|u_{1}^{2}+u_{2}^{2}|\leq1$.

Fix some coordinate system $O\xi\eta\zeta$. The position of $S$
is described by $(\xi,\eta,0)$. The position of the pendulum $B$
is described by $\left(x_{1},x_{2}\right)$: $x_{1}=\angle CSB$ is
the angle between $SB$ and $O\eta\zeta$, $x_{2}=\angle ASB$ is
the angle between $SB$ and $O\xi\zeta$.\textbf{ }Coordinates $\left(x_{1},x_{2}\right),\,-\frac{\pi}{2}\leq x_{1,2}\leq\frac{\pi}{2}$
are considered locally in the vicinity of the upper equilibrium position.
They are related to the standard spherical coordinates as follows:
\[
x_{1}=\mathrm{arcsin}(\sin\theta\sin\varphi),\, x_{2}=\mathrm{arcsin}(\sin\theta\cos\varphi)\]
where $\theta$ is a zenith angle, $0\leq\theta\leq\pi$, $\varphi$
is an azimuth angle, $0\leq\varphi\leq2\pi$.

We define the generalized coordinates and generalized forces: \[
q=\left(\xi,\eta,x_{1},x_{2}\right),\quad Q=\left(u_{1},u_{2},0,0\right)\]
  Using the Euler-Lagrange equations  we obtain nonlinear equations
of motion of the spherical inverted pendulum with a moving support
point  \cite{lob-1}\begin{align}
(M+m)\ddot{\xi}+ml\ddot{x}_{1}\cos x_{1}-ml\dot{x}_{1}^{2}\sin x_{1} & =u_{1}\nonumber \\
(M+m)\ddot{\eta}+ml\ddot{x}_{2}\cos x_{2}-ml\dot{x}_{2}^{2}\sin x_{2} & =u_{2}\label{eq:mot-with_u}\\
\ddot{\xi}b_{11}+\ddot{x}_{1}b_{13}+\ddot{x}_{2}b_{14}+\dot{x}_{1}^{2}a_{11}+\dot{x}_{2}^{2}a_{12}+\dot{x}_{1}\dot{x}a_{13}+c_{1} & =0\nonumber \\
\ddot{\eta}b_{22}+\ddot{x}_{1}b_{23}+\ddot{x}_{2}b_{24}+\dot{x}_{1}^{2}a_{21}+\dot{x}_{2}^{2}a_{22}+\dot{x}_{1}\dot{x_{2}}a_{23}+c_{2} & =0\nonumber \end{align}
where $g$ is the acceleration of gravity, $m$ is the mass of $B$,
$M$ is the mass of $S$, $b_{ij},\, a_{ij},\, c_{i}$ are some functions
of $\cos x_{k}$ and $\sin x_{k}$.

We assume that the initial state of the system is in a sufficiently
small neighbourhood of the upper unstable equilibrium position \begin{equation}
x_{1}=\dot{x}_{1}=x_{2}=\dot{x}_{2}\equiv0\label{eq:upper position}\end{equation}
 We study the minimization problem of the mean square deviation of
the pendulum from~(\ref{eq:upper position}) over an infinite time
interval:

\[
\int\limits _{0}^{\infty}(x_{1}^{2}(t)+x_{2}^{2}(t))\, dt\to\min\]
 By linearizing~(\ref{eq:mot-with_u})  about~(\ref{eq:upper position})
and eliminating the variables $\xi$ and $\eta$ we get\begin{equation}
\ddot{x}_{1}=\frac{M+m}{Ml}gx_{1}-\frac{1}{Ml}u_{1},\quad\quad\ddot{x}_{2}=\frac{M+m}{Ml}gx_{2}-\frac{1}{Ml}u_{2}\label{eq:lin-sys2}\end{equation}
 We rewrite the linearized system~(\ref{eq:lin-sys2}) in matrix
form $\ddot{x}=Kx-\left(Ml\right)^{-1}u$, where $K=kE$, $k=g\frac{M+m}{Ml}$
and $E$ is the identity matrix. Without loss of generality we assume
that $Ml=1$. Thus we come to the two-input control-affine problem.

\section{Optimal control problem. Main result }

\emph{Problem 1 }(\emph{P1})\emph{.} \[
\int\limits _{0}^{\infty}\langle x\left(t\right),x\left(t\right)\rangle\, dt\to\mathrm{inf}\]
\[
\dot{x}=y,\quad\dot{y}=Kx+u,\quad\left\Vert u\left(t\right)\right\Vert \leq1\]
\[
x(0)=x^{0},\quad y(0)=y^{0}.\]
Here $x,\, y,\, u\in\mathbb{R}^{2}$, $K$ is a $2\times2$ diagonal
matrix, $\langle\cdot,\cdot\rangle$ and $\left\Vert \cdot\right\Vert $
are the scalar product and the standard Euclidean norm on $\mathbb{R}^{2}$.
Throughout the paper we assume that the matrix $K$ is an arbitrary
non-degenerate diagonal matrix.  

The following theorem is our main result for \emph{P1.}

\begin{theorem} \label{mainTh} There exist a family of optimal
solutions of Problem 1 that are spiral-like \begin{eqnarray*}
x^{*}\left(t\right) & = & k_{x}\,\mExpLn 2T\,\left(1+g_{x}\left(T-t\right)\right),\quad\\
y^{*}\left(t\right) & = & k_{y}\,\mExpLn{\null}T\,\left(1+g_{y}\left(T-t\right)\right),\\
u^{*}(t) & = & k_{u}\,\uExpLn T\,\left(1+g_{u}\left(T-t\right)\right),\end{eqnarray*}
and all its possible rotations and reflections. They hit the origin
in a finite time $T$ (hitting time) making a countable number rotations.
Here\begin{align*}
\, & \alpha=\pm\sqrt{5},\quad k_{x},\, k{}_{y},\, k{}_{u}\in\mathbb{C},\quad i^{2}=-1,\\
\, & g_{x,y,u}\left(T-t\right)=o\left(T-t\right)^{\sigma},\,\,\sigma>0,\quad\mbox{as}\quad t\rightarrow T-0.\end{align*}
\end{theorem} Hereinafter we use the complex notation for vectors
in $\mathbb{R}^{2}$: \[
Re^{i\varphi}=\left(R\cos\varphi,R\sin\varphi\right)\]
In the next section, we give some definitions and results obtained
earlier \cite{lob-1,Zel-Bor1994,Zel-1997} for \emph{P1} and which
we will use to prove the main result.

\subsection{Optimal solutions and singular control }

It was proved \cite{Zel-Bor1994} that for any $\left(x^{0},y^{0}\right)$
from a small enough neigbourhood of the origin, there exists a unique
solution in \emph{P1. }The optimal solutions hit the origin in finite
time $T$ which  is a continuous function of $\left(x^{0},y^{0}\right)$
and satisfies \[
C_{1}\max\{\sqrt{|x^{0}|},|y^{0}|\}\leq T(x^{0},y^{0})\leq C_{2}\max\{\sqrt{|x^{0}|},|y^{0}|\}\]
 for some positive constant $C_{1}$ and $C_{2}$. It turns out that
the optimal control $u_{opt}\left(t\right)$ has no limit as $t\rightarrow T-0$.
This irregular control behaviour is due to the fact that the origin
in this problem is a singular point. 

To define a singular mode we apply Pontryagin's maximum principle
to~\emph{P1}.  If $\left(x\left(t\right),y\left(t\right),u_{opt}\left(t\right)\right)$
is an optimal solution, then there exist continuous $\mathbb{R}^{2}$-valued
functions $\phi\left(t\right)$, $\psi\left(t\right)$ and a nonnegative
constant $\lambda_{0}$ such that\[
\dot{\phi}=-\frac{\partial H}{\partial x}=\lambda_{0}x-K\psi,\quad\dot{\psi}=-\frac{\partial H}{\partial y}=-\phi,\]
\begin{equation}
\dot{x}=\frac{\partial H}{\partial\phi}=y,\quad\dot{y}=\frac{\partial H}{\partial\psi}=Kx+u_{opt}\label{eq:gamilt-syst2}\end{equation}
 Here $H$ is the Hamiltonian\[
H(x,y,\phi,\psi)=-\frac{\lambda_{0}}{2}\langle x,x\rangle+\langle y,\phi\rangle+\langle Kx,\psi\rangle+\langle u,\psi\rangle\]
The optimal control $u_{opt}\left(t\right)$ is determined by the
maximum condition:\[
H\left(x\left(t\right),\phi\left(t\right),\psi\left(t\right),u_{opt}\left(t\right)\right)=\max_{\left\Vert u\left(t\right)\right\Vert \leq1}H\left(x\left(t\right),\phi\left(t\right),\psi\left(t\right),u\right)=\,\]
\begin{equation}
=-\frac{\lambda_{0}}{2}\langle x,x\rangle+\langle y,\phi\rangle+\langle Kx,\psi\rangle+\max_{\left\Vert u\left(t\right)\right\Vert \leq1}\langle u,\psi\rangle\label{eq:max-u}\end{equation}
It can be shown that $\lambda_{0}\ne0$ for (\ref{eq:gamilt-syst2})-(\ref{eq:max-u}).
In what follows we assume  $\lambda_{0}=1$. From (\ref{eq:max-u}),
we obtain $u_{opt}\left(t\right)=\psi\left(t\right)/\left\Vert \psi\left(t\right)\right\Vert $
if $\psi\left(t\right)\neq0$. If $\psi=0$, then any admissible control
meets (\ref{eq:max-u}).

Denote $z_{1}=\psi,\, z_{2}=-\phi,\, z_{3}=-x,\, z_{4}=-y.$ In the
coordinates $z=\left(z_{1},z_{2},z_{3,}z_{4}\right)\in\mathbb{R}^{8}$
the system (\ref{eq:gamilt-syst2})--(\ref{eq:max-u}) becomes  \begin{equation}
\dot{z}_{1}=z_{2},\quad\dot{z}_{2}=z_{3}+Kz_{1},\label{eq:ham-z}\end{equation}
\[
\dot{z}_{3}=z_{4},\quad\dot{z}_{4}=-u+Kz_{3},\quad u=z_{1}/\left\Vert z_{1}\right\Vert \]

A solution $z\left(t\right)$ of~(\ref{eq:ham-z}) is called \emph{a
singular}\textbf{ }one on an interval $(t_{1},t_{2})$, if~$z_{1}\left(t\right)=0$
for all $t\in(t_{1},t_{2})$. For (\ref{eq:ham-z}) $z\left(t\right)=0$
is the unique singular solution \cite{Zel-1997}. 

If $x^{0}$ and $y^{0}$ lie in one eigenspace of the matrix $K$,
then \emph{P1} is reduced to the perturbed Fuller problem \cite{lob-1}
with a scalar control. In this case the optimal trajectory attains
the origin in a finite time, the corresponding optimal control does
not have a limit when time tends to the moment of junction of the
nonsingular arc with the singular one. And the optimal control has
an infinite number of switchings in a finite time interval (\emph{chattering}
control) \cite{lob-1}.  

In present paper we prove the existence of optimal spiral-similar
solutions of \emph{P1} that attain the singular point $z=0$ in a
finite time making a countable number of rotations. We give the main
ideas of the proof of this result. We consider a model problem for
which optimal logarithmic spirals were found \cite{Zel-Bor1994,Chukanov}.
We prove that in the neighbourhood of the origin the behaviour of
optimal solutions of \emph{P1} is determined by optimal solutions
of the \emph{model problem}. For this we apply Pontryagin's maximum
principle to Problem 1 and to the model problem and obtain two Hamiltonian
systems. Then we use the blow-up method in the origin. Analyzing the
dynamics of both blown-up Hamiltonian systems, we get that they coincide
on the image of the origin. Moreover we show that there exist periodic
trajectories which are solutions of both systems and we prove that
the periodic orbits are hyperbolic and their stable manifolds are
woven in logarithmic spirals. Finally we perform the blow-down and
get the main results for our problem.

\subsection{Model problem}

Let $K=0$. Then \emph{Problem 1} takes the form

\emph{Problem 2 }(\emph{P2})\emph{.} \[
\int\limits _{0}^{\infty}\langle x\left(t\right),x\left(t\right)\rangle\, dt\to\mathrm{inf}\]
\[
\dot{x}=y,\quad\dot{y}=u,\quad\left\Vert u\left(t\right)\right\Vert \leq1\]
\[
x(0)=x^{0},\quad y(0)=y^{0}.\]
 The Hamiltonian system for \emph{P2} is as follows \begin{equation}
\begin{aligned}\dot{z}_{1} & =z_{2}, &  &  & \dot{z}_{2} & =z_{3}\\
\dot{z}_{3} & =z_{4}, &  &  & \dot{z}_{4} & =-u, &  &  & u & =z_{1}/\left\Vert z_{1}\right\Vert \end{aligned}
\label{eq:model-hamilton_syst}\end{equation}
For \emph{Problem 2}, all the results formulated in the previous
section hold. And besides it was found \cite{Zel-Bor1994,Chukanov}
that the system (\ref{eq:model-hamilton_syst}) has solutions in the
form of logarithmic spirals \begin{eqnarray}
z_{m\zeta}^{*}(t) & = & -\zeta A_{m-1}\mExpLn{5-m}{T^{*}},\quad m=\overline{1,4},\label{eq:log-spir}\\
u_{\zeta}^{*}(t) & = & -\zeta\uExpLn{T^{*}},\quad0\leq t<T^{*}\nonumber \end{eqnarray}
 Here$\,\zeta\in\mathcal{SO}(2),\,\, i^{2}=-1,\,\,\alpha=\pm\sqrt{5},\,\, A_{0}=-\frac{1}{126},\,\, A_{m+1}=-A_{m}(4-m+i\alpha),$
$m=0,1,2$. The trajectories (\ref{eq:log-spir}) hit the origin in
a finite time $T^{*}$, and the optimal control $u^{*}(t)$ performs
an infinite number of rotations along the circle $S^{1}$.  We will
show that there are similar optimal logarithmic spirals when $K\neq0$. 

\begin{theorem} \label{th22} In a sufficiently small neighbourhood
of the origin there exist the following solutions of (\ref{eq:ham-z}):\begin{eqnarray*}
z_{m}(t) & = & k_{m}\,\mExpLn{5-m}T\,\left(1+g_{m}\left(T-t\right)\right),\quad m=\overline{1,4},\\
u(t) & = & k_{0}\,\uExpLn T\,\left(1+g_{0}\left(T-t\right)\right),\qquad\end{eqnarray*}
where $k_{m}\in\mathbb{C}$, $g_{m}\left(T-t\right)=o\left(T-t\right)^{\sigma},\,\,\sigma>0,$
as $t\rightarrow T-0$. \end{theorem} Theorem \ref{mainTh} is then
an obvious corollary of Theorem \ref{th22}.

%
{}

\subsection{Blowing up the singularity}

To prove Theorem \ref{th22} we use the procedure of resolution of
singularity for the Hamiltonian system (\ref{eq:ham-z}) \cite{Zel-Bor1994,Z-L-H_2017}.
We use the same scheme as in \cite{Z-L-H_2017} and the similar change
of coordinates. Consider the blowing up the singularity at the origin
by the map $B:\, z\mapsto(\mu,\tilde{z})$: \begin{eqnarray}
\tilde{z}_{4}=\frac{z_{4}}{\mu},\quad\tilde{z}_{3}=\frac{z_{3}}{\mu^{2}},\quad\tilde{z}_{2}=\frac{z_{2}}{\mu^{3}},\quad\tilde{z}_{1}=\frac{z_{1}}{\mu^{4}},\label{eq:zz}\\
\mu=\left(\left|\frac{z_{4}}{A_{3}}\right|^{24}+\left|\frac{z_{3}}{A_{2}}\right|^{12}+\left|\frac{z_{2}}{A_{1}}\right|^{8}+\left|\frac{z_{1}}{A_{0}}\right|^{6}\right)^{\frac{1}{24}}\nonumber \end{eqnarray}
 where $\mu\in\mathbb{R}_{+}$, $A_{0}=-1/126,\, A_{j+1}=-A_{j}(4-j+i\alpha),\,\alpha=\sqrt{5},\, j=0,1,2$,
and $\tilde{z}\in\mathbb{R}^{8}$ lies on the manifold \[
\Pi=\left\{ \left|\frac{\tilde{z}_{4}}{A_{3}}\right|^{24}+\left|\frac{\tilde{z}_{3}}{A_{2}}\right|^{12}+\left|\frac{\tilde{z}_{2}}{A_{1}}\right|^{8}+\left|\frac{\tilde{z}_{1}}{A_{0}}\right|^{6}=1\right\} .\]
 Let $Q$ denote the cylinder $\Pi\times\{\mu\in\mathbb{R}\}$, and
$Q_{0}=Q\cap\{\mu=0\}$. It was proved \cite{Z-L-H_2017} that  $B$
is a diffeomorphism from $\mathbb{R}^{8}\setminus\left\{ 0\right\} $
onto $Q\cap\left\{ \mu>0\right\} $. In the coordinates $(\mu,\tilde{z})$
the system (\ref{eq:ham-z}) has the form: \begin{align}
\dot{\mu} & =\mathcal{M}(\mu,\tilde{z}) &  &  & u & =\tilde{z}_{1}/|\tilde{z}_{1}|\nonumber \\
\dot{\tilde{z}}_{1} & =\frac{1}{\mu}(\tilde{z}_{2}-4\tilde{z}_{1}\mathcal{M}) &  &  & \dot{\tilde{z}}_{3} & =\frac{1}{\mu}(\tilde{z}_{4}-2\tilde{z}_{2}\mathcal{M})\label{eq:razd-1}\\
\dot{\tilde{z}}_{2} & =\frac{1}{\mu}(\tilde{z}_{3}+\mu^{2}\tilde{z}_{1}-3\tilde{z}_{2}\mathcal{M}) &  &  & \dot{\tilde{z}}_{4} & =\frac{1}{\mu}(u+\mu^{2}K\tilde{z}_{3}-\tilde{z}_{4}\mathcal{M})\nonumber \end{align}
 where\begin{align}
\mathcal{M}(\mu,\tilde{z}) & =\left(\frac{1}{|A_{3}|^{24}}|\tilde{z}_{4}|^{22}\langle\tilde{z}_{4},\frac{\tilde{z}_{1}}{|\tilde{z}_{1}|}\rangle+\frac{1}{2|A_{2}|^{12}}|\tilde{z}_{3}|^{10}\langle\tilde{z}_{3},\tilde{z}_{2}\rangle+\null\right.\nonumber \\
 & \phantom{=\left(\frac{1}{24}\right.}+\frac{1}{3|A_{1}|^{8}}|\tilde{z}_{2}|^{6}\langle\tilde{z}_{2},\tilde{z}_{3}\rangle+\frac{1}{4|A_{0}|^{6}}|\tilde{z}_{1}|^{4}\langle\tilde{z}_{1},\tilde{z}_{2}\rangle+\null\label{eq:M}\\
{} & \phantom{=\left(\frac{1}{24}\right.}+\left.\mu^{2}\frac{1}{|A_{3}|^{24}}|\tilde{z}_{4}|^{22}\langle\tilde{z}_{4},K\tilde{z}_{3}\rangle+\mu^{2}\frac{1}{3|A_{1}|^{8}}|\tilde{z}_{2}|^{6}\langle\tilde{z}_{2},K\tilde{z}_{1}\rangle\right)\nonumber \end{align}
Denote\begin{align*}
\mathcal{M}_{0}(\tilde{z}) & =\left(\frac{1}{|A_{3}|^{24}}|\tilde{z}_{4}|^{22}\langle\tilde{z}_{4},\frac{\tilde{z}_{1}}{|\tilde{z}_{1}|}\rangle+\frac{1}{2|A_{2}|^{12}}|\tilde{z}_{3}|^{10}\langle\tilde{z}_{3},\tilde{z}_{2}\rangle+\null\right.\\
 & \phantom{=\left(\frac{1}{24}\right.}+\left.\frac{8}{|A_{1}|^{8}}|\tilde{z}_{2}|^{6}\langle\tilde{z}_{2},\tilde{z}_{3}\rangle+\frac{6}{|A_{0}|^{6}}|\tilde{z}_{1}|^{4}\langle\tilde{z}_{1},\tilde{z}_{2}\rangle\right)\\
\mathcal{M}_{1}(\tilde{z}) & =\left(\frac{1}{|A_{3}|^{24}}|\tilde{z}_{4}|^{22}\langle\tilde{z}_{4},K\tilde{z}_{3}\rangle+\frac{1}{3|A_{1}|^{8}}|\tilde{z}_{2}|^{6}\langle\tilde{z}_{2},K\tilde{z}_{1}\rangle\right)\end{align*}
   then $\mathcal{M}(\mu,\tilde{z})=\mathcal{M}_{0}(\tilde{z})+\mu^{2}\mathcal{M}_{1}(\tilde{z})$.
For \emph{P2} $\left(K=0\right)$ we have $\mathcal{M}_{1}(\tilde{z})=0$.

We define a new time parametrization by \begin{equation}
ds=\frac{1}{\mu}dt\label{eq:s-t}\end{equation}
 After this reparametrization system (\ref{eq:razd-1}) becomes smooth\begin{align}
\mu' & =\mu\mathcal{M}, &  &  & u & =\tilde{z}_{1}/|\tilde{z}_{1}|\nonumber \\
\tilde{z}_{1}' & =\tilde{z}_{2}-4\tilde{z}_{1}\mathcal{M}, &  &  & \tilde{z}_{3}' & =\tilde{z}_{4}-2\tilde{z}_{3}\mathcal{M},\label{eq:bl-notmod}\\
\tilde{z}_{2}' & =\tilde{z}_{3}+\mu^{2}\tilde{z}_{1}-3\tilde{z}_{2}\mathcal{M}, &  &  & \tilde{z}_{4}' & =u+\mu^{2}K\tilde{z}_{3}-\tilde{z}_{4}\mathcal{M},\nonumber \end{align}
When $K=0$,  one has $\mathcal{M}_{1}(\tilde{z})=0$ and (\ref{eq:bl-notmod})
turns into\begin{align}
\mu' & =\mu\mathcal{M}_{0}(\tilde{z}), &  &  & u & =\tilde{z}_{1}/|\tilde{z}_{1}|\nonumber \\
\tilde{z}_{1}' & =\tilde{z}_{2}-4\tilde{z}_{1}\mathcal{M}_{0}(\tilde{z}), &  &  & \tilde{z}_{3}' & =\tilde{z}_{4}-2\tilde{z}_{3}\mathcal{M}_{0}(\tilde{z}),\label{eq:bl-mod}\\
\tilde{z}_{2}' & =\tilde{z}_{3}-3\tilde{z}_{2}\mathcal{M}_{0}(\tilde{z}), &  &  & \tilde{z}_{4}' & =u-\tilde{z}_{4}\mathcal{M}_{0}(\tilde{z}),\nonumber \end{align}
Thus, the system (\ref{eq:bl-notmod}) is a small perturbation of
(\ref{eq:bl-mod}).

\subsection{Periodic solution }

Consider the particular solution (\ref{eq:log-spir}) for $\zeta=E$,
where $E$ is the identity $2\times2$-matrix \begin{eqnarray}
{z}_{m}^{*}(t) & := & {z}_{mE}^{*}(t)=-A_{m-1}\mExpLn{5-m}{T^{*}},\quad m=\overline{1,4},\label{eq:log-spir-1}\\
{u}^{*}(t) & := & u_{E}^{*}(t)=-\uExpLn{T^{*}},\nonumber \end{eqnarray}
Rewrite (\ref{eq:log-spir-1}) in coordinates $(\mu,\tilde{z})$ \begin{align}
\mu^{*}(t) & =T^{*}-t, & \mathcal{M}_{0} & =-1, & \tilde{u}^{*}(t)=-e^{i\alpha\ln|T^{*}-t|},\nonumber \\
\tilde{z}_{m}^{*}(t) & =-A_{m-1}e^{i\alpha\ln|T^{*}-t|}, & m & =\overline{1,4},\label{eq:cycle*}\end{align}
Passing in~(\ref{eq:cycle*}) to the parameter $s$, we get a solution
of (\ref{eq:bl-mod})\begin{align*}
\mu^{*}\left(s\right) & =T^{*}e^{-s}, & \tilde{u}^{*}(s) & =-e^{-i\alpha s},\\
\tilde{z}_{m}^{*}(s) & =-A_{m-1}e^{-i\alpha s},\quad m=\overline{1,4}.\end{align*}

Denote $\xi\left(s\right)=\left(\mu\left(s\right),\,\tilde{z}(s)\right)\in\mathbb{R}^{9}$.
Consider $\xi^{0}\left(s\right)=\left(0,\,\tilde{z}^{*}(s)\right)$.
Since for $\mu=0$ the systems (\ref{eq:bl-mod}) and (\ref{eq:bl-notmod})
coincide, $\xi^{0}\left(s\right)$ is a~periodic solution (a~cycle)
for both systems (\ref{eq:bl-notmod}) and (\ref{eq:bl-mod}). We
will study the behaviour of solutions of (\ref{eq:bl-notmod}) in
the neighbourhood of $\xi^{0}\left(s\right)$. We will show that $\xi^{0}\left(s\right)$
is a hyperbolic cycle.

\begin{utv}\label{cycle} The periodic solution $\xi^{0}\left(s\right)$
has exactly one characteristic exponent with negative real part, seven
characteristic exponents with positive real part, and exactly one
characteristic exponent equals zero.

\end{utv}

\begin{proof} 

 To prove this, we linearize (\ref{eq:bl-notmod}) and (\ref{eq:bl-mod})
 around $\xi^{0}\left(s\right)$. Since the right sides of (\ref{eq:bl-notmod})
and (\ref{eq:bl-mod}) differ by $\bar{o}\left(\mu\right)$ we get
the same equations of variation \begin{equation}
h'=F_{\xi}\left(\xi^{0}\left(s\right)\right)h\label{eq:var1}\end{equation}
where $F$ denotes the right side of (\ref{eq:bl-mod}) and $F_{\xi}$
is the Jacobian matrix of $F$. Here we will not write $F_{\xi}\left(\xi^{0}\left(s\right)\right)$
explicitly, since it is too cumbersome. Since $F$ does not explicitly
depend on $s$, and $\xi^{0}\left(s\right)$ is a periodic trajectory,
then $F_{\xi}\left(\xi^{0}\left(s\right)\right)$ a continuous periodic
matrix.  By Floquet theory  there exists \cite{Farkas} a periodic
invertible state-space transformation $\hat{h}=P\left(s\right)h$
(which is called the Lyapunov transformation) such that in the new
coordinates the system (\ref{eq:var1}) has the form $\hat{h}'=J\hat{h}$,
where $J$ is a constant matrix. Moreover, the eigenvalues of the
matrix $J$ will be characteristic exponents of the cycle~$\xi^{0}\left(s\right)$.
It is difficult in general to find $P\left(s\right)$ explicitly,
however one can guess the matrix $P$ in our case:  \[
P_{9\times9}=\left(\begin{array}{ccccc}
1 & 0 & 0 & 0 & 0\\
0 & P_{1} & 0 & 0 & 0\\
0 & 0 & P_{2} & 0 & 0\\
0 & 0 & 0 & P_{3} & 0\\
0 & 0 & 0 & 0 & P_{4}\end{array}\right),\quad P_{j}=\left(\begin{array}{cc}
\cos\alpha & -\sin\alpha\\
\sin\alpha & \cos\alpha\end{array}\right),\,\, j=\overline{1,4}\]
Here $0$ are zero matrices of corresponding sizes. 

A straightforward computation yields the matrix $J$ to be {\small 
\[ J =
\begin{pmatrix}
-1 &0 & 0 & 0 & 0 & 0 & 0 & 0 & 0\\ 
0 & 24 & -4\sqrt 5 & \frac{316}{63} & \frac{100\sqrt 5}{63} & -\frac{67}{63} & -\frac{73\sqrt 5}{63} & -\frac{53}{63} & \frac{47\sqrt 5}{63} \\ 
0 & \sqrt 5 & 4 & 0 & -1 & 0 & 0 & 0 & 0 \\ 
0 & 60 & -9\sqrt 5 & \frac{442}{21} & \frac{79\sqrt 5}{21} & -\frac{46}{21} & -\frac{73\sqrt 5}{21} & -\frac{53}{21} & \frac{47\sqrt 5}{21} \\ 
0 & 15\sqrt 5 & -\frac{45}{4} & \frac{463\sqrt 5}{84} & \frac{188}{21} & -\frac{67\sqrt 5}{84} & -\frac{281}{84} & -\frac{53\sqrt 5}{84} & \frac{235}{84} \\ 
0 & -70 & \frac{21\sqrt 5}{2} & -\frac{379}{18} & -\frac{50\sqrt 5}{9} &\frac{103}{18} & \frac{55\sqrt 5}{18} & \frac{71}{18} & -\frac{47\sqrt 5}{18} \\ 
0 & -70\sqrt 5 & \frac{105}{2} & -\frac{379\sqrt 5}{18} & -\frac{250}{9} &\frac{85\sqrt 5}{18} & \frac{401}{18} & \frac{53\sqrt 5}{18} & -\frac{217}{18} \\ 
0 & -105 & \frac{63\sqrt 5}{4} & -\frac{379}{12} & -\frac{25\sqrt 5}{3} &\frac{67}{12} & \frac{73\sqrt 5}{12} & \frac{65}{12} & -\frac{59\sqrt 5}{12} \\ 
0 & 105\sqrt 5 & \frac{189}{4} & \frac{379\sqrt 5}{12} & \frac{125}{3} & -\frac{67\sqrt 5}{12} & -\frac{365}{12} & -\frac{41\sqrt 5}{12} & \frac{247}{12}  
\end{pmatrix}
\]
}
 
The characteristic polynomial of $J$ is \[
(\lambda+1)\lambda(\lambda-4)(\lambda-5)(\lambda-93)\left((\lambda-5)^{2}\lambda^{2}+36(\lambda-5)\lambda+630\right)=0\]
 Thus \begin{eqnarray*}
\lambda_{1} & = & -1,\quad\,\lambda_{2}=0,\quad\lambda_{3}=4,\quad\,\lambda_{4}=5,\,\quad\lambda_{5}=93\\
\lambda_{6,7} & = & \frac{1}{2}\left(5+\sqrt{47\pm12\sqrt{34}i}\right)\approx4.65903\pm4.0511i\\
\lambda_{8,9} & = & \frac{1}{2}\left(5-\sqrt{47\pm12\sqrt{34}i}\right)\approx0.340974\pm4.0511i\end{eqnarray*}
are the eigenvalues of $J$. \end{proof} %
\begin{figure}
\begin{centering}
\includegraphics[scale=0.98]{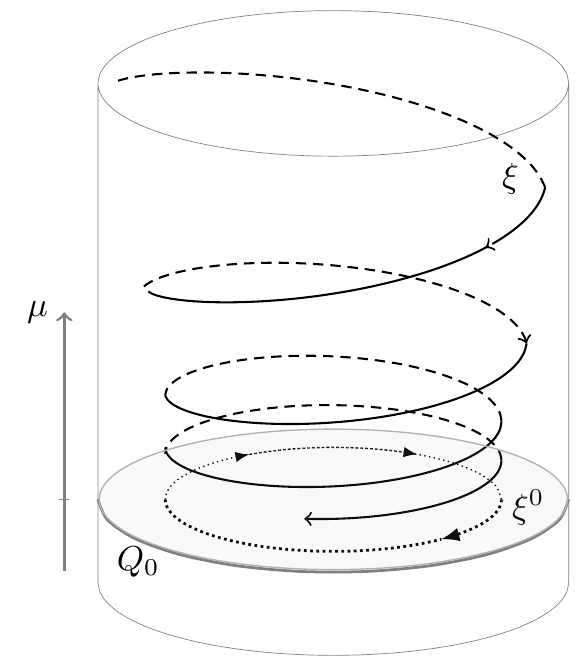}
\par\end{centering}

\caption{solutions of the blown-up Hamiltonian system that lie on $Q$ and
tend to $\xi^{0}$ }

\end{figure}

 Therefore $\xi^{0}(s)$ is not orbitally stable, and there are the
contraction in the $\mu$-direction and the expansion in the other
directions \cite{Farkas}.  Applying the invariant manifolds theorem
\cite{Hartman} for $\xi^{0}(s)$ we obtain that  there exist solutions
 of both (\ref{eq:bl-mod}) and (\ref{eq:bl-notmod}) satisfying
\begin{equation}
\left\Vert \xi(s+s_{0})-\xi^{0}(s)\right\Vert e^{cs}\to0\,\,\text{ as }s\to\infty\label{eq:hartman1}\end{equation}
 for some $s_{0}\,$ and $c>0$. (\ref{eq:hartman1}) can be written
as follows \begin{align}
|\mu(s+s_{0})|e^{cs} & \to0\nonumber \\
|\tilde{z}_{m}(s+s_{0})-\tilde{z}_{m}^{*}(s)|e^{cs} & \to0,\quad\, m=\overline{1,4}.\label{eq:hartman}\end{align}
 Thus we construct the two-dimensional stable manifolds of $\xi^{0}(s)$
for (\ref{eq:bl-mod}) and (\ref{eq:bl-notmod}) (Fig.~2). In the
next section we prove that these stable manifolds are woven in logarithmic
spirals.

\subsection{Proof the main result}

Let $\xi\left(s\right)=\left(\mu\left(s\right),\tilde{z}\left(s\right)\right)$
be a solution of (\ref{eq:bl-notmod}) that satisfies (\ref{eq:hartman1}).
Denote by $T$ the time to reach the origin along $\xi\left(s\right)$. 

\begin{utv}\label{utv27} 

$\mu\left(s\right)=\kappa e^{-s}\left(1+\bar{o}\left(e^{-c_{\mu}s}\right)\right),\quad s\to+\infty,$
where $\kappa,\,\, c_{\mu}$ are some positive constants.

\end{utv}

\begin{proof} 

Let $\xi^{*}\left(s\right)=\left(\mu^{*}\left(s\right),\tilde{z}^{*}\left(s\right)\right)$
and $\xi\left(0\right)=\xi^{*}\left(0\right)$.    Rewrite the
first equations of systems (\ref{eq:bl-notmod}) and (\ref{eq:bl-mod})

 \[
\left\{ \begin{aligned}\left(\mu(s)\right)^{-1}\frac{d}{ds}\mu(s) & =\mathcal{M}(\mu(s),\tilde{z}(s))\\
\left(\mu^{*}(s)\right)^{-1}\frac{d}{ds}\mu^{*}(s) & =\mathcal{M}_{0}(\mu^{*}(s),\tilde{z}^{*}(s))\end{aligned}
\right.\]
Note that $\mathcal{M}_{0}(\mu^{*}(s),\tilde{z}^{*}(s))=-1$. Consider
the difference between the above equations \[
\left(\mu(s)\right)^{-1}\frac{d}{ds}\mu(s)-\left(\mu^{*}(s)\right)^{-1}\frac{d}{ds}\mu^{*}(s)=\mathcal{M}(\mu(s),\tilde{z}(s))-\mathcal{M}_{0}(\mu^{*}(s),\tilde{z}^{*}(s))\]
 Hence 

\begin{equation}
\frac{d}{ds}\ln\left(\frac{\mu(s)}{\mu^{*}(s)}\right)=\mathcal{M}(\mu(s),\tilde{z}(s))-\mathcal{M}_{0}(\mu^{*}(s),\tilde{z}^{*}(s))\label{eq:lnmu}\end{equation}
 By (\ref{eq:M}) and (\ref{eq:hartman}) we have for some $c_{\mu}>0$
\begin{equation}
\left|\mathcal{M}(\mu(s),\tilde{z}(s))-\mathcal{M}_{0}(\mu^{*}(s),\tilde{z}^{*}(s))\right|e^{c_{\mu}s}\to0\quad\text{ as }s\to\infty\label{eq:M-infty}\end{equation}
 or \[
\mathcal{M}(\mu(s),\tilde{z}(s))=\mathcal{M}_{0}(\mu^{*}(s),\tilde{z}^{*}(s))+\bar{o}\left(e^{-c_{\mu}s}\right)\quad\text{ as }s\to\infty\]
  (\ref{eq:M-infty}) implies \[
q_{*}=\int_{0}^{\infty}\left(\mathcal{M}(\mu(v),\tilde{z}(v))-\mathcal{M}_{0}(\mu^{*}(v),\tilde{z}^{*}(v))\right)dv<\infty\]
Thence we have \begin{multline*}
\int_{0}^{s}\left(\mathcal{M}(\mu(v),\tilde{z}(v))-\mathcal{M}_{0}(\mu^{*}(v),\tilde{z}^{*}(v))\right)dv=\\
\begin{split}\, & =q{}_{*}-\int_{s}^{\infty}\left(\mathcal{M}(\mu(v),\tilde{z}(v))-\mathcal{M}_{0}(\mu^{*}(v),\tilde{z}^{*}(v))\right)dv=\,\\
 & =q_{*}-\int_{s}^{\infty}\bar{o}\left(e^{-c_{\mu}v}\right)dv=q_{*}+\bar{o}\left(e^{-c_{\mu}s}\right)\quad\text{ as }s\to\infty.\end{split}
\end{multline*}
   Integrating (\ref{eq:lnmu})  and taking into account $\ln\left(\frac{\mu(0)}{\mu^{*}(0)}\right)=0$
we obtain  \begin{eqnarray}
\ln\left(\frac{\mu(s)}{\mu^{*}(s)}\right)=q_{*}+\bar{o}\left(e^{-c_{\mu}s}\right), & s\to\infty\label{eq:ln}\end{eqnarray}
 By exponentiating (\ref{eq:ln}), we go to \begin{eqnarray*}
\frac{\mu(s)}{\mu^{*}(s)}=e^{q_{*}}\left(1+\bar{o}\left(e^{-c_{\mu}s}\right)\right), & s\to\infty\end{eqnarray*}
 Since $\mu^{*}\left(s\right)=T^{*}e^{-s}$ we get \[
\mu\left(s\right)=\kappa e^{-s}\left(1+\bar{o}\left(e^{-c_{\mu}s}\right)\right),\quad s\to+\infty,\]
 where $\kappa=T^{*}e^{q_{*}}$. \end{proof}

Using the definition (\ref{eq:s-t}) of $s\left(t\right)$ and Lemma
\ref{utv27} we get  \[
T-t=\int_{s\left(t\right)}^{\infty}\mu\left(s\right)ds=\kappa e^{-s\left(t\right)}\left(1+\bar{o}\left(e^{-c_{\mu}s\left(t\right)}\right)\right),\,\, t\to T-0\]
Hence \begin{align}
e^{-s\left(t\right)} & =\kappa^{-1}\left(T-t\right)\left(1+\bar{o}\left(e^{-c_{\mu}s\left(t\right)}\right)\right)^{-1}=\label{eq:bound-e-s}\\
 & =\kappa^{-1}\left(T-t\right)\left(1+\bar{o}\left(e^{-c_{\mu}s\left(t\right)}\right)\right),\quad\, t\to T-0\nonumber \end{align}
Denote $g_{1}\left(t\right)=\kappa^{-1}\left(1+\bar{o}\left(e^{-c_{\mu}s\left(t\right)}\right)\right)$.
In view of (\ref{eq:bound-e-s}) it is seen that \begin{equation}
g_{1}\left(t\right)=\kappa^{-1}\left(1+\bar{o}\left(\left(T-t\right){}^{c_{\mu}}\right)\right)\,\,\text{ as }t\to T-0\label{eq:bound-g1}\end{equation}

Using (\ref{eq:bound-e-s})--(\ref{eq:bound-g1}) and taking into
account $a^{\omega}=e^{\omega\ln a}\,\,\left(a>0,\,\,\omega\mathbb{\in C}\right)$,
we obtain  \begin{align}
e^{-i\alpha s(t)} & =e^{i\alpha\ln\left((T-t)g_{1}\left(t\right)\right)}=\nonumber \\
 & =e^{i\alpha\ln(T-t)}e^{i\alpha\ln g_{1}\left(t\right)}=\label{eq:ist}\\
 & =e^{i\alpha\ln(T-t)}e^{-i\alpha\ln\kappa}\left(1+\bar{o}\left(\left(T-t\right){}^{c_{\mu}}\right)\right),\qquad t\to T-0\nonumber \end{align}

Consider $\xi\left(s\right)=\left(\mu\left(s\right),\tilde{z}\left(s\right)\right)$
satisfying (\ref{eq:hartman}). Put \[
\tilde{z}_{m}(s)=\Phi_{m}\left(s\right)e^{i\theta_{m}\left(s\right)},\quad m=1,\ldots,4.\]
 Lemma \ref{utv27} and (\ref{eq:hartman}) imply that  in the neighbourhood
of $\xi^{0}(s)$ \begin{align*}
\Phi_{m}\left(s\right) & =\left|A_{m-1}\right|\left(1+\bar{o}\left(e^{-cs}\right)\right),\\
\theta_{m}\left(s\right) & =\pi-\alpha s+\beta+\bar{o}\left(e^{-cs}\right),\,\quad\, s\to\infty\end{align*}
 where $\beta=\alpha s_{0}\in\mathbb{R}$, $s_{0}$ is the same as
in (\ref{eq:hartman}), $A_{m-1}\left(m=1,\ldots,4\right)$ are defined
in (\ref{eq:log-spir}). Hence \begin{equation}
\tilde{z}_{m}(s)=-B_{m-1}e^{-i\alpha s}(1+\bar{o}\left(e^{-cs}\right)),\,\quad\, s\to\infty\label{eq:z-s}\end{equation}
 where $B_{m-1}=A_{m-1}e^{i\alpha s_{0}}\in\mathbb{C},\,\, m=1,\ldots,4.$ 

Let $s\left(t\right)$ be determined by (\ref{eq:s-t}). Setting $s=s\left(t\right)$
in (\ref{eq:z-s}) and taking into account  (\ref{eq:bound-e-s})--(\ref{eq:ist})
we obtain \begin{equation}
\tilde{z}_{m}(t)=D_{m-1}e^{i\alpha\ln\left(T-t\right)}(1+\bar{o}((T-t)^{b})),\,\,\label{eq:zz-zz}\end{equation}
 as $t\to T-0.$ Here $b>0,\,\, D_{m-1}=-B_{m-1}e^{-i\alpha\ln\kappa}\in\mathbb{C},\,\, m=1,\ldots,4$.
Using (\ref{eq:zz}) and (\ref{eq:zz-zz}) we get \begin{eqnarray*}
z_{m}(t) & = & D_{m-1}(T-t)^{m-1}e^{i\alpha\ln(T-t)}(1+\bar{o}((T-t)^{\sigma})),\,\quad\, t\to T-0\end{eqnarray*}
 This proves Theorem \ref{th22}, and therefore Theorem \ref{mainTh}.

\subsection*{Acknowledgments }

The authors would like to thank Prof. M.I. Zelikin and Prof. L.V.
Lokutsievskiy for helpful discussions. 


\end{document}